\documentclass [11pt] {article}
\usepackage{amsfonts}
\usepackage{amssymb}
\usepackage{times}
\usepackage{graphicx}

\newtheorem{lemma}{Lemma}[section]
\newtheorem{corollary}{Corollary}[section]
\newtheorem{theorem}{Theorem}

\newtheorem{question}{Question}
\newtheorem{remark}{Remark}[section]

\newtheorem{proposition}{Proposition}[section]
\newtheorem{definition}{Definition}[section]

\def\ran{{\rm ran}}

\def\<{\langle}
\def\>{\rangle}

\def\to{\rightarrow}

\def\rad{{\rm rad}}

\def\spec{{\rm spec}}

\def\hth{{\mathcal{H}\otimes\mathcal{H}^*}}

\def\rk{{\rm rk}}
\begin{document}

\title{Limits of kernel operators and the spectral regularity lemma}
\author{{\sc Bal\'azs Szegedy}}

\maketitle

\abstract{We study the spectral aspects of the graph limit theory. We give a description of graphon convergence in terms of converegnce of eigenvalues and eigenspaces. Along these lines we prove a spectral version of the strong regularity lemma. Using spectral methods we investigate group actions on graphons. As an application we show that the set of isometry invariant graphons on the sphere is closed in terms of graph convergence however the analogous statement does not hold for the circle. This fact is rooted in the representation theory of the orthogonal group.}

\tableofcontents

\section{Introduction}

The so called graph limit theory (see \cite{FLS},\cite{LSz1},\cite{BCLSV1},\cite{BCL},\cite{LSz3},\cite{LSz6},\cite{LSz7},\cite{LSz8}) is a type of calculus developed on the completion of the set of finite graphs.
As it was proved in \cite{LSz1}, the elements of the completion can be represented by two variable symmetric functions $W:[0,1]^2\rightarrow[0,1]$. It is very natural to interpret such functions as self adjoint integral kernel operators on $L_2([0,1])$. Classical theory says that every such operator has a spectral decomposition converging in $L_2$. In this paper we focus on the spectral aspects of the graph limit theory. In the course of this investigation various interesting topics come up.

It was proved in \cite{LSz3} that the graph limit space is compact in the topology generated by a distance (called $\delta_\square$) derived from the well known cut norm. This compactness implies a strong form $\cite{AFKS}$ of Szemer\'edi's regularity lemma \cite{Szem2}.
We give a new interpretation of cut norm convergence and $\delta_\square$ convergence in terms of spectral decompositions. Roughly speaking we prove that a sequence is convergent if and only if the eigenvalue sequences and the eigenspace structures converge in a rather strong way.

As a consequence we obtain a spectral form of the strong regularity lemma which can be regarded as a generalization of the strong regularity lemma by Alon, Fischer, Krivelevich and Szegedy \cite{AFKS}. We mention that numerous spectral aspects of the regularity lemma were studied by several authors. The closest approach to ours is by Frieze and Kannan \cite{FK}. An advantage of this type of regularization is that it is invariant under the symmetry group of the graph or graphon.

Using this fact we show that graphons or graphs can be regularized in a way that the structured part, which is a step function with bounded number of steps, is approximatively invariant under every automorphism of the graphon or graph. We call this statement the ``symmetry preserving regularity lemma''. A symmetry preserving removal lemma was proved in \cite{Sz1}

If a unitary group action $G$ on $L_2([0,1])$ stabilizes a given graphon then the eigenspaces are also invariant under $G$. In particular they define finite dimensional representations of $G$.
This creates an interesting connection between regularization and representation theory.
In \cite{Gow} Gowers proved that if in a finite group the dimension of the minimal non trivial irreducible representation is sufficiently big then its Cayley graphs are all arbitrarily quasi random. It is not hard to generalize this result for graphons with unitary group actions (see corollary \ref{qra2}). In the infinite case however a new interesting phenomenon appears. Let $(V,\mu)$ be a probability space. If a unitary action of $G$ on $L_2(V,\mu)$ satisfies the condition that for every $k$ there is a finite dimensional subspace of $L_2(V,\mu)$ containing all the $G$ invariant subspaces of dimension at most $k$ then we say the $G$ acts weakly random.
It turns out that graphons invariant under a weakly random action behave in a more controlled way.
For example we prove that they form a closed set in the cut norm and so in the $\delta_\square$ distance. In particular the set of such graphons has a graph theoretic characterization using inequalities in subgraph densities.

Quite surprisingly, the circle behaves very different from the higher dimensional spheres. The set of isometry invariant graphons on the spheres of dimension $\geq 2$ is closed in the $\delta_\square$-norm and so it has a ``graph theoretic characterization''.
On the other hand isometry invariant graphons on the circle can have limits which can only be defined on the torus (or some other compact abelian group).
This fundamental difference is coming from the fact that the action of $O_{n+1}$ on $L_2(S_n)$ is weakly random if and only if $n\geq 3$.

\subsection{Hilbert-Schmidt kernel operators}

Let us fix a separable probability space $(V,\mu)$. Let $\mathcal{H}$ denote the complex Hilbert space $L_2(V,\mu)$ with scalar product $(f,g)=\mathbb{E}_v(\overline{f(v)}g(v))$.
The elements of $\mathcal{H}$ are measurable functions so it makes sense to talk about their $L_\infty$ (or $L_1$) norms even though the $L_\infty$ norm might be infinite.

Recall that a sequence $\{f_i\}_{i=1}^\infty$ is called weakly convergent if $\{(f_i,g)\}_{i=1}^\infty$ is convergent for every $g\in\mathcal{H}$. It follows from the principle of uniform boundedness that weakly convergent sequences are bounded. Every weakly convergent sequence has a limit in $\mathcal{H}$. It is easy to see that every bounded sequence has a weakly convergent subsequence. It is known that convex bounded closed sets are weakly compact. For example the closed unit ball in the $L_\infty$ norm is weakly compact.

\begin{lemma}\label{trivweak} Let $\{f_i\}$ be a weakly convergent sequence in $\mathcal{H}$ with limit $f$ such that $\lim_{i\to\infty}\|f_i\|_2=\|f\|_2$ then $\{f_i\}_{i=1}^\infty$ converges to $f$ in the $L_2$ norm.
\end{lemma}

\begin{proof} We have $\lim_{i\to\infty}\|f_i-f\|_2=\lim_{i\to\infty}(f-f_i,f-f_i)=\lim_{i\to\infty}\|f\|_2+\|f_i\|_2-(f,f_i)-(f_i,f)=0$.
\end{proof}

\bigskip

A function $M:V\times V\rightarrow\mathbb{C}$ is called a Hilbert-Schmidt kernel operator if $M\in L_2(V\times V,\nu\times\nu)$. This is equivalent with saying that $M\in\hth$.
The operator $M$ acts on $\mathcal{H}$ by $f\mapsto Mf$ where $Mf(x)=\mathbb{E}_y(M(x,y)f(y))$.
The image space $\ran(M)$ of $M$ is the Hilbert space generated by the functions $\{Mf|f\in\mathcal{H}\}$.

We will use the notion of weak convergence of kernel operators. A sequence of Hilbert-Schmidt kernel operators $\{H_i\}_{i=1}^\infty$ is called weakly convergent if they are weakly convergent in the Hilbert space $\hth=L_2(V\times V)$.
It is easy to see that if $\{\|H_i\|_2\}_{i=1}^\infty$ is a bounded sequence then $\{H_i\}_{i=1}^\infty$ is weakly convergent if and only if
the sequences $\{g^*H_if\}_{i=1}^\infty$ are convergent for every pair $f,g\in L_2(V)$.

An important consequence of the Cauchy-Schwartz inequality is that
\begin{equation}\label{caeq}
\|Mg\|_2\leq\|M\|_2\|g\|_2~~~~{\rm and}~~~~|f^*Mg|\leq\|f\|_2~\|g\|_2~\|M\|_2
\end{equation}
for every $f,g\in\mathcal{H}$ and $M\in\hth$.
We will need the next lemma.

\begin{lemma}\label{oneqeak} Let $\{f_i\}_{i=1}^\infty$ be a weakly convergent sequence in $\mathcal{H}$ with limit $f$. If $M\in\hth$ then $\lim_{i=1}^\infty \|Mf_i-Mf\|_2=0$.
\end{lemma}

\begin{proof} We have that $\|f_i\|_2\leq c$ for every $i$ with some positive constant $c$. Let $\{b_i\}_{i=1}^\infty$ be an ortho-normal basis in $\mathcal{H}$. Then $M=\sum_{i,j}\alpha_{i,j}b_i^*b_j$ where $\sum_{i,j}|\alpha_{i,j}|^2=\|M\|_2$. For every $\epsilon>0$ there is a number $t$ such that $M_t=\sum_{1\leq i,j\leq t}\alpha_{i,j}b_i^*b_j$ satisfies $\|M-M_t\|_2\leq\epsilon$. We have that by (\ref{caeq}) that
$$\|Mf_i-Mf\|_2=\|(M-M_t)f_i+(M-M_t)f+M_tf_i-M_tf\|_2\leq 2\epsilon c+\|M_tf_i-M_tf\|_2.$$
If $i$ is big enough then
$$\|M_tf_i-M_tf\|^2_2=\sum_{i=1}^t\Bigl|\sum_{j=1}^t\alpha_{i,j}(b_j,f_i-f)\Bigr|^2$$ is smaller than $(\epsilon c)^2$ and for such indices $\|Mf_i-Mf\|_2\leq 3\epsilon c$.
Applying it for every $\epsilon>0$ the proof is complete.
\end{proof}

This implies immediately the next lemma.

\begin{lemma}\label{twoweak} Let $\{f_i\}_{i=1}^\infty$ and $\{g_i\}_{i=1}^\infty$ be two weakly convergent sequences in $\mathcal{H}$ with limits $f$ and $g$. Let $M$ be a Hilbert-Schmidt kernel operator. Then $\lim f_i^*Mg_i=f^*Mg$.
\end{lemma}

Let $M\in\hth$ be a self adjoint Hilbert-Schmidt kernel operator.
It is is well known that $M$ has a spectral decomposition
$$M=\sum_{i=1}^\infty f_if_i^*\lambda_i$$
where $\{f_i\}_{i=1}^\infty$ is an orthogonal system of unit length elements in $\mathcal{H}$ and the number's $\lambda_i\in\mathbb{R}$ satisfy $\sum_{i=1}^\infty|\lambda_i|^2=\|M\|_2$.
The numbers $\{\lambda_i\}_{i=1}^\infty$ are the eigenvalues of $M$ listed with multiplicities.

\begin{definition} Let $M:V\times V\rightarrow\mathbb{C}$ be a self adjoint kernel operator with spectral decomposition $M=\sum_i f_if_i^*\lambda_i$ and $\lambda\geq 0$. Then we denote by $[M]_\lambda$ the sum
$\sum_{\{i|~|\lambda_i|>\lambda\}}f_if_i^*\lambda_i$.
\end{definition}

It is easy to see that $[M]_\lambda$ does not depend on the concrete choice of the functions $f_i$ even if there are multiple eigenvalues. A basis independent definition of $[M]_\lambda$ is  $[M]_\lambda=\sum_{|\tau|>\lambda}\tau P_\tau$ where $P_\tau$ is the orthogonal projection to the eigenspace $W_\tau=\{f|Mf=\tau f\}$.

Since $M$ is a measurable function on $V\times V$ we can talk about the $L_\infty$ and $L_1$ norms of $M$. Kernel operators with finite $L_\infty$ norms will have a special importance for us.

\begin{lemma}\label{eigenmax}Let $M$ be a self adjoint kernel operator with $\|M\|_\infty\leq 1$. Assume that $f\in\mathcal{H}$ satisfies $\|f\|_2=1$ and $Mf=\lambda f$ for some non zero number $\lambda$.
Then $\|f\|_\infty\leq |\lambda|^{-1}$.
\end{lemma}

\begin{proof} By Cauchy-Schwartz we have that $|\lambda f(x)|=|Mf(x)|=|\mathbb{E}_y(M(x,y)f(y))|\leq\|f\|_2=1$.
\end{proof}

The spectral radius is an important invariant of kernel operators. It is defined as
$$\rad(M)=\sup_{\|f\|_2=1}\|Mf\|_2.$$
The spectral radius of a self adjoint kernel operator is the maximum of the absolute values of its eigenvalues.

\subsection{The cut norm}

We will use the cut norm on $\hth$ defined by
$$\|M\|_\square=\sup_{\|f\|_\infty,\|g\|_\infty\leq 1}|f^*Mg|.$$
where $f$ and $g$ ranges over all possible measurable functions on $V$ with $L_\infty$ norm at most $1$.
Note that there are several definitions of the $\|.\|_\square$-norm that equivalent up to constant constant multiples. For example in \cite{LSz1} we used $\sup_{S,T}|\int_{S\times T}M|$ where $S,T$ runs through all pairs of measurable sets in $V$.

\begin{lemma}\label{specrad} $\|M\|_\square\leq\spec(M)$.
\end{lemma}

\begin{proof} Let $f$, $g$ be arbitrary functions with $\|f\|_\infty\leq 1$ and $\|g\|_\infty\leq 1$. We have that $\|f\|_2\leq 1$ and $\|g\|_2\leq 1$. Then $\|Mg\|_2\leq\rad(M)$ and so by Cauchy-Schwartz $f^*MG\leq\rad(M)$.
\end{proof}

This implies the next lemma.

\begin{lemma}\label{cutapprox} If $M:V\times V\rightarrow\mathbb{C}$ is a self adjoint Hilbert-Schmidt operator and $\alpha>0$ then $\|M-[M]_\alpha\|_\square\leq \alpha$.
\end{lemma}

\begin{lemma}\label{cutcontriv} If a bounded sequence of kernel operators $\{M_i\}_{i=1}^\infty$ converges to $M$ in the cut norm then it converges to $M$ in the weak topology in $\hth$. In particular if $f,g\in\mathcal{H}$ then $\lim_{i\to\infty} f^*M_ig=f^*Mg$.
\end{lemma}

\begin{proof} Let $S$ be the set of finite linear combinations of operators of the form $fg^*$ where $f,g\in L_\infty(V)$. It is classical that $S$ is dense in the Hilbert space $\hth$.
The cut norm convergence implies that $\lim_{i\to\infty} (Q,M_i)=(Q,M)$ for every $Q\in S$. Since $\{M_i\}_{i=1}^\infty$ is a bounded sequence it has to be weakly convergent with limit $M$.
\end{proof}

The previous lemma with lemma \ref{trivweak} implies the next corollary.

\begin{corollary}\label{corhv} If a sequence $\{M_i\}$ in $\hth$ converges to $M$ in the cut norm and $\lim_{i\to\infty}\|M_i\|_2=\|M\|_2$ then $\{M_i\}_{i=1}^\infty$ converges to $M$ in $L_2$.
\end{corollary}

Let $\psi:V\rightarrow V$ be a measure preserving map. This means that $\psi$ is measurable and $\mu(\psi^{-1}(A))=\mu(A)$ for every measurable subset $A\subseteq V$.
If $W:V\times V\rightarrow\mathbb{C}$ is a kernel operator then we define $W^\psi$ by $W^\psi(x,y)=W(\psi(x),\psi(y))$.

Let $\|.\|_n$ be one of the norms $\|.\|_1~,~\|.\|_2~,~\|.\|_\square$ on $\hth$.
We define the distance $\delta_n(W_1,W_2)$ on $\hth$ by
$$\inf_{\psi_1,\psi_2:V\rightarrow V}\|W_1^{\psi_1}-W_2^{\psi_2}\|_n$$
where $\psi_1$ and $\psi_2$ ranges over all possible measure preserving maps on $V$.

It is easy to see (and was pointed out in several papers) that one of the maps (say $\psi_2$) can be omitted and the other one can be assumed to be invertible. This means
$$\delta_n(W_1,W_2)=\inf_\psi\|W_1^\psi-W_2\|_n$$
where $\psi$ ranges through all invertible measure preserving maps.
This fact together with lemma \ref{corhv} implies the next lemma.

\begin{lemma}\label{corhv2} Let $\{M_i\}_{i=1}^\infty$ be a $\delta_\square$-convergent sequence with limit $M$. If $\|M\|_2=\lim_{i\to\infty}\|M_i\|_2$ then $\{M_i\}_{i=1}^\infty$ converges to $M$ also in $\delta_1$.
\end{lemma}

Let $\mathcal{M}$ denote the set of self adjoint operators in $\hth$ with $L_\infty$ norm at most $1$. Let $\mathcal{X}$ be the space obtained form $\mathcal{M}$ by identifying operators that are $\delta_\square$ distance $0$ from each other. This way $(\mathcal{X},\delta_\square)$ becomes a metric space.
The next theorem follows from the results in \cite{LSz3}:

\begin{theorem}\label{compactness} The metric space $(\mathcal{X},\delta_\square)$ is compact.
\end{theorem}

\subsection{Graph limits}

Let $\mathcal{W}_0$ denote the set of symmetric measurable functions $W:[0,1]^2\rightarrow [0,1]$.
If $G=(V,E)$ is a finite simple graph on the vertex set $[k]=\{1,2,\dots,k\}$ then the homomorphism density of $G$ in $W$ is defined by
$$t(G,W)=\int_{x_1,x_2,\dots,x_n}\prod_{(i,j)\in E,i<j}W(x_i,x_j)~dx_1~dx_2\dots~dx_n$$
where $x_1,x_2,\dots,x_k$ are in $[0,1]$.

The elements of $\mathcal{W}_0$ are also called graphons. Two graphons are equivalent if their $\delta_\square$ distance is $0$.
Let $\mathcal{X}_0$ denote the set of equivalence classes of graphons. The set $(\mathcal{X}_0,\delta_\square)$ is a compact topological space. We call $\mathcal{X}_0$ the graph limit space.
The theory of graph limits is basically the calculus on the graph limit space.

A sequence $\{W_i\}_{i=1}^\infty$ in $\mathcal{W}_0$ is $\delta_\square$-convergent if and only if $\{t(G,W_i)\}_{i=1}^\infty$ is convergent for every simple graph $G$. Two graphons $W_1,W_2$ are equivalent if $t(G,W_1)=t(G,W_2)$ for every simple graph $G$. This implies that homomorphism densities are well defined on the elements of $\mathcal{X}$.
Let $\mathcal{G}$ be the set of finite simple graphs.

\begin{definition} A map $p:\mathcal{X}_0\rightarrow\mathbb{R}$ is called a graph polynomial if there are finitely many graphs $G_1,G_2,\dots,G_n$ in $\mathcal{G}$ and real numbers $\lambda_1,\lambda_2,\dots,\lambda_n$ such that
$p(W)=\sum_{i=1}^n t(G_i,W)\lambda_i$ for every $W\in\mathcal{X}_0$.
\end{definition}

If $G\in\mathcal{G}$ is the disjoint union of $G_1$ and $G_2$ then $t(G,W)=t(G_1,W)t(G_2,W)$ for every $W\in\mathcal{X}_0$. It follows that polynomials are closed under multiplications and so they are forming a commutative algebra (containing the constant functions) of $\delta_\square$ continuous functions on $\mathcal{X}_0$.
Let $K(\mathcal{X}_0)$ denote this algebra.
The Stone-Weierstrass theorem implies that every continuous function on $\mathcal{X}_0$ can be approximated in $L_\infty$ by some polynomial function in $K(\mathcal{X}_0)$.
The next lemma shows that closed subsets of $\mathcal{X}_0$ can be characterized through inequalities in subgraph densities.

\begin{lemma} A set $C\subseteq\mathcal{X}_0$ is a closed subset of $\mathcal{X}_0$ then there is a countable set of graph polynomials $\{p_i\}_{i=1}^\infty$ such that $C=\cap_{i=1}^\infty \{W|p_i(W)\geq 0\}$.
\end{lemma}

\begin{proof} Let $d(W)=\inf\{\delta_\square(W,W_2)|W_2\in C\}$ be the distance function from $C$. Since $d$ is a continuous function we can approximate it with arbitrary $L_\infty$ precision by graph polynomials. Let $p'_i$ be an $1/i$ approximation of $d$ in $K(\mathcal{X}_0)$ and let $p_i=1/i-p'_i$. It is clear that $\{p_i\}_{i=1}^\infty$ is an appropriate system of polynomials.
\end{proof}

\subsection{Convergence in cut norm}

In this part we examine the relationship between spectral decompositions and convergence in cut norm. We use the notation from the previous chapter. Let $\{M_i\}_{i=1}^\infty$ be a sequence of self adjoint kernel operators in $\hth$ with $\|M_i\|_\infty\leq 1$ such that they converge in the cut norm. Let $M$ be the cut norm limit of $\{M_i\}_{i=1}^\infty$. Obviously $M$ is a self adjoint kernel operator and satisfies $\|M\|_\infty\leq 1$.
We will keep this notation for the rest of this chapter and we will prove statements on the properties of the sequence $\{M_i\}_{i=1}^\infty$.

It is not hard to see that cut norm convergence implies the following type of convergence of the spectrums.
If for a kernel operator $W$ with spectrum $\{\lambda_i\}_{i=1}^\infty$ we define the random variable $X(W)$ that takes the value $\lambda_i$ with probability $\lambda_i^4(\sum_i\lambda_i^4)^{-1}$ then the $k$-th moment of $X(W)$ is equal to $t(C_{4+k},W)/t(C_4,W)$ where $C_n$ is the cycle of length $n$. It follows that from \cite{LSz1} that $\{X(M_i)\}_{i=1}^\infty$ converges to $X(M)$ in the weak topology of probability distributions.
The spectrum with multiplicities is fully decodable from $X(W)$ and $t(C_4,W)$ and so at the level of spectrums the cut norm convergence is fully described.

In the rest of the chapter we study joint convergence of the spectrum and the eigenspaces.

\begin{lemma}\label{weakeigen} Let $\{f_i\}_{i=1}^\infty$ be a weakly convergent sequence in $\mathcal{H}$ with limit $f$ such that $\|f_i\|_2=1$ for every $i$ and $M_if_i=f_i\lambda_i$ where $\lim_{i\to\infty}\lambda_i=\lambda\neq 0$. Then $\{f_i\}_{i=1}^\infty$ converges in $L_2$ to $f$ and $Mf=\lambda f$.
\end{lemma}

\begin{proof} Note that lemma \ref{eigenmax} implies that $\|f_i\|_\infty\leq|\lambda_i|^{-1}$. By lemma \ref{twoweak} we have that $\lim_{i\to\infty} f_i^*Mf_i=f^*Mf$.
On the other hand
$$|f_i^*(M-M_i)f_i|\leq |\lambda_i|^{-2}\|M-M_i\|_\square$$ and from $\lambda\neq 0$ we get
\begin{equation}\label{eq1}
0=\lim_{i\to\infty}|f_i^*(M-M_i)f_i|=\lim_{i\to\infty}|f_i^*Mf_i-\lambda_i|=|f^*Mf-\lambda|.
\end{equation}
Using $\|f_i\|_\infty\leq|\lambda_i|^{-1}$ we get
\begin{equation}\label{eq2}
\lim_{i\to\infty}|f^*(M_i-M)f_i|\leq\lim_{i\to\infty}\|f\|_\infty|\lambda_i|^{-1}\|M-M_i\|_\square=0.
\end{equation}
It follows by (\ref{eq2}), lemma \ref{twoweak} and by (\ref{eq1}) that
$$\lim_{i\to\infty}\lambda_i(f,f_i)=\lim_{i\to\infty}f^*M_if_i=\lim_{i\to\infty}f^*Mf_i=f^*Mf=\lambda$$
and so $\lim_{i\to\infty}(f,f_i)=1$. Since $\|f_i\|_2=1$ and $\|f\|_2\leq 1$ this is only possible if $\|f\|_2=1$ and so by lemma \ref{trivweak} $\{f_i\}_{i=1}^\infty$ converges to $f$ in $L_2$.

Now we need to show that $Mf=\lambda f$. Let $g$ be any element in $\mathcal{H}$ with $(g,f)=0$. We have that $$g^*M_if=g^*M_i(f-f_i)+g^*M_if_i=g^*M_i(f-f_i)+\lambda_i(g,f_i)\leq$$
$$\leq \|g\|_2\|M_i\|_2\|f-f_i\|_2+\lambda_i(g,f_i).$$ It implies that $\lim_{i\to\infty} g^*M_if=0.$ On the other hand by lemma \ref{cutcontriv} $\lim_{i\to\infty}g^*M_if=g^*Mf$. It follows that $Mf$ is orthogonal to every function $g$ which is orthogonal to $f$. It follows that $f$ is an eigenvalue of $M$ and by $f^*Mf=\lambda$ and $\|f\|_2=1$ the proof is complete.
\end{proof}

\begin{lemma}\label{speccon1} Let $\lambda>0$ be a number such that $\{-\lambda,\lambda\}\cap\spec(M)=\emptyset$.
Then
\begin{enumerate}
\item $\lim_{i\to\infty}\rk([M_i]_\lambda)=\rk([M]_\lambda)$,
\item $\lim_{i\to\infty}\|[M_i]_\lambda-[M]_\lambda\|_2=0$.
\end{enumerate}
\end{lemma}

\begin{proof} Assume that $M_i=\sum_{j=1}^\infty f_{i,j}f_{i,j}^*\lambda_{i,j}$
such that $\{|\lambda_{i,j}|\}_{j=1}^\infty$ is a decreasing sequence and the vectors $\{f_{i,j}\}_{\lambda_{i,j}\neq 0}$ are forming an orthonormal system. To keep the sequences $\{\lambda_{i,j}\}_{j=1}^\infty$ infinite we put an infinite number of $0$'s at the end if $M_i$ has finite rank. (If $\lambda_{i,j}=0$ then $f_{i,j}$ is an arbitrarily chosen function of unit length.)

First of all we prove that there is a subsequence $\{M_i\}_{i\in S}$ satisfying the condition of the lemma. By a standard argument we can choose a subsequence $S$ such that $\{f_{i,j}\}_{i\in S}$ is weakly convergent for every fixed $j$ and $\{\lambda_{i,j}\}_{i\in S}$ is convergent for every $j$. Let $f_j$ be the weak limit of $\{f_{i,j}\}_{i\in S}$ and $\lambda_j$ be the limit of $\{\lambda_{i,j}\}_{i\in S}$. Obviously we have that $\sum_{j=1}^\infty|\lambda_j|^2\leq 1$ and $\{|\lambda_j|\}_{j=1}^\infty$ is a decreasing sequence. It follows that $|\lambda_j|\leq 1/\sqrt{j}$. First of all note that if $\lambda_j\neq 0$ then by lemma \ref{weakeigen} $\lim_{i\to\infty}\|f_{i,j}-f_i\|_2=0$. It follows that if $\lambda_{j_1}$ and $\lambda_{j_2}$ ar both non-zero then $(f_{j_1},f_{j_2})=0$. In other words $\{f_j\}_{\{j|\lambda_j\neq 0\}}$ is an orthonormal system of functions. Let
$$M'=\sum_{\{j|\lambda_j\neq 0\}}f_jf_j^*\lambda_j.$$

First we claim that $M=M'$. For every natural number $t$ we have that $\|M_i-M'\|_\square$ is at most
$$\|\sum_{j=1}^t\Bigl(f_{i,j}f_{i,j}^*\lambda_{i,j}-f_jf_j^*\lambda_j\Bigr)\|_\square+\|\sum_{j=t+1}^\infty f_{i,j}f_{i,j}^*\lambda_{i,j}\|_\square+\|\sum_{j=t+1}^\infty f_jf_j^*\lambda_j\|_\square .$$
The spectral radius of the sums in the last two terms is at most $1/\sqrt{t+1}$. It follows
that if $i$ is big enough then $\|M_i-M'\|_\square\leq 3/\sqrt{t+1}$.
By letting $t$ go to infinity we get that $M_i$ converges to $M'$ in the cut norm and so $M'=M$.

Let $t$ be an integer greater than $\lambda^{-2}$. We have that $|\lambda_{i,j}|\leq\lambda$ whenever $j>t$. Since $\lambda,-\lambda$ are not eigenvalues of $M$ we have that there is an index $i_0$ such that for $|\lambda_{i,j}-\lambda_j|\leq |\lambda-|\lambda_j||/2$ whenever $1\leq j\leq t$ and $i>i_0~,~i\in S$. This means that for such indices $|\{j||\lambda_{i,j}|>|\lambda|\}|=|\{j|\lambda_j>|\lambda|\}|=\rk([M]_\lambda)$ showing that $\rk([M_i]_\lambda)=\rk([M]_\lambda)$.

Now we finish the general case by contradiction. If the first statement is not true then we can choose an infinite subsequence where $\rk([M_i]_\lambda)\neq [M]_\lambda$. This is a contradiction sice from such a subsequence we can not choose a sub sequence satisfying the first condition.
If the second condition fails then we can choose an infinite subsequence for some $\epsilon$ such that $\|[M_i]_\lambda-[M]_\lambda\|_2>\epsilon$. This is again a contradiction
\end{proof}

\subsection{Spectral Regularity lemma}

\begin{theorem}[Spectral regularity lemma]\label{specreg} For an arbitrarily decreasing function $F:\mathbb{\mathbb{R}^+\times\mathbb{R}^+}\rightarrow\mathbb{R}^+$ and every $\epsilon>0$ there is a constant $\delta>0$ such that for every self adjoint kernel operator $M:V\times V\rightarrow\mathbb{C}$ with $\|M\|_\infty\leq 1$ on a separable probability space $(V,\mu)$ there is a real number $\lambda\geq\delta$ such that $M$ has a decomposition
$M=S+E+R$ with the following properties
\begin{enumerate}
\item $S=[M]_\lambda$
\item $\|E\|_2\leq\epsilon$
\item $\|R\|_\square\leq F(\lambda,\epsilon)$
\item $\|S+E\|_\infty\leq 1$
\item $E$ and $R$ are self adjoint.

\end{enumerate}
\end{theorem}

\begin{proof} We go by contradiction. Let $\epsilon>0$ be a real number such that the theorem fails for $\epsilon$. This means that there is a sequence of kernel operators $\{M_i\}_{i=1}^\infty$ with $L_\infty$ norm at most $1$ such that $M_i$ does not have the desired decomposition for $\delta=1/i$.  We can assume without loss of generality that all the operators $M_i$ are defined on the same standard probability space $V$. Also without loss of generality (by choosing a subsequence guaranteed by theorem \ref{compactness}) we can assume that $\{M_i\}_{i=1}^\infty$ is convergent in $\delta_\square$ and so there is a sequence of invertible measure preserving maps $\{\psi_i\}_{i=1}^\infty$ on $V$ such that $\{M_i^{\psi_i}\}_{i=1}^\infty$ converges to $M$ with $\|M\|_\infty\leq 1$ in the cut norm.
Let $\lambda>0$ be a number such that $\{\lambda,-\lambda\}\cap\spec{M}=\emptyset$ and $\|[M]_\lambda-M\|_2\leq\epsilon/3$. By lemma \ref{speccon1} there is an index $i_0$ such that for $i>i_0$ we have $\|[M_i^{\psi_i}]_\lambda-[M]_\lambda\|_2\leq\epsilon/3$. This means that if $i>i_0$ the $\|M-[M_i^{\psi_i}]_\lambda\|_2\leq2\epsilon/3$. Let $E_i=M-[M_i^{\psi_i}]_\lambda$ and $R_i=M^{\psi_i}_i-M$. Now $M^{\psi_i}_i=[M^{\psi_i}_i]_\lambda+E_i+R_i$. Now since $R_i$ converges to $0$ in the cut norm it follows that there is an index $i_1>\max(i_0,1/\lambda)$ such that if $i>i_1$ then $\|R_i\|_\square<F(\lambda,\epsilon)$ satisfies the theorem with $1/i$. Applying $\psi^{-1}$ to the decomposition of $M_i^{\psi_i}$ we get a contradiction.
\end{proof}

Now let us assume that $V$ is a finite probability space with uniform distribution.
We can represent undirected graphs on the vertex set $V$ by their adjacency matrices $G:V\times V\rightarrow\{0,1\}$. More generally assume that $G:V\times V\rightarrow\mathbb{R}$ is a symmetric matrix. The automorphism group ${\rm Aut}(G)$ is the group of permutation matrices $g$ satisfying $gGg^{-1}=G$. A matrix $G:V\times V\rightarrow\mathbb{R}$ is called a step function with $n$-steps if there is a partition $\mathcal{P}=\{P_1,P_2,\dots,P_n\}$ of $V$ and an $n$ by $n$ matrix $T:[n]^2\rightarrow\mathbb{R}$ such that $G(a,b)=T(i,j)$ whenever $a\in P_i$ and $b\in P_j$.
We say that $G$ is a balanced step function with $n$ steps if $||P_i|-|P_j||\leq 1$ for every $i,j$.

\begin{remark} Let $C\subseteq\mathcal{H}$ be a convex, $L_\infty$-bounded closed set which is invariant under measure preserving maps on $V$. Then similar regularity lemma holds for every $M$ in $C$ such that $S+E\in C$. In particular if $C$ is the set of kernel operators taking values in $[0,1]$ then $M$ is a graphon and so is $S+E$. The proof is essentially the same.
\end{remark}

\begin{lemma}[Eigenvector clustering] If $\epsilon>0$ and $G=\sum_{i=1}^k f_if_i^*\lambda_i$ such that $\|f_i\|_\infty\leq m$ and $|\lambda_i|\leq m$ for every $1\leq i\leq k$ then there is a step function $T$ with at most $(20km^3/\epsilon)^k$ steps such that $\|T-G\|_\infty\leq\epsilon$.
\end{lemma}

\begin{proof} Let us consider the partition $V=\cup_{i=1}^t P_i$ according to the level sets of the function $$v\mapsto (\lfloor f_i(v)\epsilon_1^{-1}\rfloor\epsilon_1)_{i=1}^k.$$
Here $t\leq (2m/\epsilon_1)^k$
If two element $v_1,v_2\in P_i$ and $w_1,w_2\in P_j$ then with a rough estimate $|T(v_1,w_1)-T(v_2,w_2)|\leq 10km^2\epsilon_1$. It follows that there is step function $T$ with partition $\{P_i\}_{i=1}^t$ such that $\|T-G\|_\infty\leq 10km^2\epsilon_1$. If $\epsilon_1\leq\epsilon/(10km^2)$ then $T$ satisfies the condition of the lemma.
\end{proof}

Using the fact that in theorem \ref{specreg} the matrix $[M]_\lambda$ is invariant under the automorphisms of $G$ and the previous lemma we obtain the next version of the classical graph regularity lemma.

\begin{theorem}[Symmetry preserving regularity lemma] For an arbitrarily decreasing function $F(\mathbb{R}^+,\mathbb{N})\rightarrow\mathbb{R}^+$ and $\epsilon>0$ there is constant $n$ such that for every symmetric matrix $G:V\times V\rightarrow [-1,1]$ there is decomposition
$G=S+E+R$ such that
\begin{enumerate}
\item $S$ is a step function with $s\leq n$ steps
\item $\|gSg^{-1}-S\|_\infty\leq\epsilon$ for every $g\in{\rm Aut}(G)$
\item $\|E\|_2\leq\epsilon$
\item $\|R\|_\square\leq F(\epsilon,s)$.
\end{enumerate}
\end{theorem}

At the cost of worsening the bound of $n$ in terms of $F$ and $\epsilon$ we can also assume that $T$ is a balanced step function. However in this case the $L_\infty$ error in $\|gSg^{-1}-S\|_\infty\leq\epsilon$ becomes an $L_2$ error.

\subsection{Eigenspace convergence}

Let $\mathcal{M}$ be the set of kernel operators $M:V\times V\rightarrow\mathbb{C}$ with $\|M\|_\infty\leq 1$.

\begin{proposition}\label{spacecon1} Let $\{M_i\}_{i=1}^\infty$ be a sequence in $\mathcal{M}$. Then the following two statements are equivalent.
\begin{enumerate}
\item $\{M_i\}_{i=1}^\infty$ is convergent in the cut norm
\item there is a decreasing positive real sequence $\{\alpha_i\}_{i=1}^\infty$ with $\lim_{i=1}^\infty\alpha_i=0$ such that $\{[M_i]_{\alpha_j}\}_{i=1}^\infty$ is $L_2$ convergent for every $j$.
\end{enumerate}
Furthermore in the second statement the cut norm limit $M$ of $\{M_i\}_{i=1}^\infty$ can be computed as $$M=\lim_{j\to\infty}(\lim_{i\to\infty} [M_i]_{\alpha_j})$$ converging in $L_2$.
\end{proposition}

\begin{proof} To show that the first statement implies the second one let $S$ be the set of eigenvalues of the cut norm limit $M$ of $\{M_i\}_{i=1}^\infty$. Then by lemma \ref{speccon1} any sequence $\{\alpha_i\}_\infty$ avoiding the absolute values of the eigenvalues of $M$ satisfies the convergence requirement second statement. The eigenvalues of $M$ are forming a countable set and so we can choose $\{\alpha_i\}_{i=1}^\infty$ satisfying the required conditions.

Let $L_j=\lim_{i\to\infty}[M_i]_{\alpha_j}$ and $D_j=L_{j+1}-L_j$.
The functions $D_{i,j}=[M_i]_{\alpha_{j+1}}-[M_i]_{\alpha_j}$ satisfy the following properties for every $i$
\begin{enumerate}
\item $D_{i,j}D_{i,k}=0$ for every $i\neq k$,
\item $(D_{i,j},D_{i,k})=0$ in $\hth$,
\item the spectral radius of $D_{i,j}$ is at most $\alpha_j$,
\item $\sum_{j=1}^\infty \|D_{i,j}\|_2^2=\|M_i\|_2^2\leq 1$.
\end{enumerate}

This means that the sequence $\{D_j\}_{j=1}^\infty$ satisfies the same properties.
Using that last and the second property we have that $\sum_{j=1}^\infty D_j$ is convergent in the $L_2$ norm. Let us denote the limit by $M$. Our goal is to show that $\{M_i\}_{i=1}^\infty$ converges to $M$ in the cut norm.

The first and third property implies that the spectral radius of $N_j=\sum_{k=j}^\infty D_k$ is at most $\alpha_j$ and so
$$\|M-M_i\|_\square\leq\|(M-N_{j+1})-[M_i]_{\alpha_j}\|_\square+\|M_i-[M_i]_{\alpha_j}\|_\square+\|N_{j+1}\|_\square\leq$$
$$\leq\|L_j-[M_i]_{\alpha_j}\|_2+2\alpha_j.$$
Now using the fact that $[M_i]_{\alpha_j}$ converges to $L_j$ we get that if $i$ is big enough then $\|M-M_i\|_\square\leq 3\alpha_j$. Since $\{\alpha_j\}_{i=1}^\infty$ converges to $0$ the proof is complete.
\end{proof}

\begin{lemma}\label{limsuplem} If $\{M_i\}_{i=1}^\infty$ in $\mathcal{M}$ converges to $M$ in the $\delta_\square$ distance then for every $\lambda>0$ we have $\limsup_{i\to\infty}\|[M]_\lambda\|_2\leq \|M\|_2$.
\end{lemma}

\begin{proof} We can choose a sequence of invertible measure preserving maps $\{\psi_i\}_{i=1}^\infty$ such that $\{M_i^{\psi_i}\}_{i=1}^\infty$ is convergent in the cut norm. By lemma \ref{speccon1} there is a value $0<\alpha<\lambda$ such that $\lim_{i\to\infty}[M_i]_\alpha=[M]_\alpha$ in $L_2$. This means that $\limsup_{i\to\infty}\|[M_i]_\lambda\|_2\leq\lim_{i\to\infty}\|[M_i]_\alpha\|_2\leq\|M_2\|_2$.
\end{proof}

\begin{proposition} Let $\{M_i\}_{i=1}^\infty$ be a sequence in $\mathcal{M}$. Then the following two statements are equivalent.
\begin{enumerate}
\item $\{M_i\}_{i=1}^\infty$ is convergent in the $\delta_\square$ distance
\item there is a decreasing positive real sequence $\{\alpha_i\}_{i=1}^\infty$ with $\lim_{i=1}^\infty\alpha_i=0$ such that $\{[M_i]_{\alpha_j}\}_{i=1}^\infty$ is $\delta_1$ convergent for every $j$.
\end{enumerate}
Furthermore in the second statement the cut norm limit $M$ of $\{M_i\}_{i=1}^\infty$ can be computed as $$M=\lim_{j\to\infty}(\lim_{i\to\infty} [M_i]_{\alpha_j})$$ converging in $\delta_1$.
\end{proposition}

\begin{proof} To see that the first statement implies the second choose a sequence of invertible measure preserving transformations $\{\psi_i\}_{i=1}^\infty$ such that $\{M_i^{\psi_i}\}_{i=1}^\infty$ converges in the cut norm. Then proposition \ref{spacecon1} shows the second statement.

We show that the second statement implies the first. $L_j=\lim_{i\to\infty}[M_i]_{\alpha_j}$. Observe that by lemma \ref{cutapprox} $\|M_i-[M_i]_{\alpha_j}\|_\square\leq\alpha_j$ and so $\delta_\square(M_i,[M_i]_{\alpha_j})\leq\alpha_j$. If $i$ is big enough that $\delta_\square(M_i,L_j)\leq 2\alpha_j$. This means that $\{M_i\}_{i=1}^\infty$ is a $\delta_\square$ Cauchy sequence which shows the first statement.
It also shows that $\{L_j\}_{j=1}^\infty$ converges to the cut norm limit $M$ of $\{M_i\}_{i=1}^\infty$ in the $\delta_{\square}$-distance.
It remains to show that this convergence is also true in the $\delta_1$ metric.
This follows from the fact that by lemma \ref{limsuplem} $\|L_j\|_2\leq\|M\|_2$ and so lemma \ref{corhv2} completes the proof.
\end{proof}

\subsection{Regularization and group actions}

An advantage of the spectral regularity lemma is that it is invariant under the symmetries of kernel operators.
To be more precise let $G$ be a group of unitary operators on $\mathcal{H}=L_2(V)$.
Then there is a natural induced action of $G$ on $\hth$ and in particular on the set of self adjoint Hilbert-Schmidt kernel operators. This action satisfies $H^\alpha f^\alpha=(Hf)^\alpha$ and $(f^\alpha)^*H^\alpha=(f^*H)^\alpha$. If $H=\sum_i f_if_i^*\lambda_i$ is a spectral decomposition of $H$ then $H^\alpha=\sum_i f_i^\alpha(f_i^\alpha)^*\lambda_i$.
The next well known lemma is trivial from the previous remarks.

\begin{lemma}\label{inveigen} If a group $G$ of unitary operators on $L_2(V)$ stabilizes a self adjoint Hilbert-Schmidt kernel operator $H$ ($H^\alpha=H$ for every $\alpha\in G$) then $G$ stabilizes all the operators $[H]_\lambda$ for $\lambda\geq 0$.
Furthermore the eigenspaces $W_{\lambda_i}$ of $H$ are $G$ invariant spaces. In particular $\ran(H)$ is $G$ invariant.
\end{lemma}

A typical example for a kernel operator stabilized by a group action is a Cayley graphon.
Let $G$ be a compact Hausdorff topological group with normalized Haar measure $\mu$. Let $f$ be a Borel measurable function $f:G\rightarrow\mathbb{C}$. Let $M:G\times G\rightarrow\mathbb{C}$ be the kernel operator defined by $M(x,y)=f(x^{-1}y)$. It is easy to see that the left action of $G$ on itself induces a unitary group action of $G$ on $L_2(G,\mu)$ and it stabilizes $M$. If $f$ has the property that $f(g^{-1})=\overline{f(g)}$ then the corresponding Cayley graphon is self adjoint.
The next corollary of lemma \ref{inveigen} creates a connection between quasi randomness and representation theory.

\begin{corollary}[Quasirandom action]\label{qra1} If a group $G$ of unitary operators on $L_2(V)$ stabilizes a self adjoint Hilbert-Schmidt operator $H$ with $\|H\|_2\leq 1$ then the spectral radius (and so the cut norm) of $H$ is at most $1/\sqrt{d}$ where $d$ is the smallest dimension of a $G$ invariant subspace in $\ran(H)$.
\end{corollary}

\begin{proof} Lemma \ref{inveigen} implies that every eigenvalue of $H$ has multiplicity at least $d$. Since the sum of the squares of the eigenvalues is at most $1$ we get that $\lambda^2d<1$ is satisfied by every eigenvalue $\lambda$.
\end{proof}

\begin{corollary}[Quasirandom action II.]\label{qra2} Let $G$ be a compact Hausdorff topological group with normalized Haar measure $\mu$. Let $K\subseteq G$ be a closed subgroup and $V$ be the left coset space $\{gK|g\in G\}$. Let $\mathcal{H}=L_2(V,\mu)$ and $d$ be the degree of the smallest non trivial representation of $G$ which appears in the induced action of $G$ on $\mathcal{H}$.
Then every $G$ invariant self adjoint Hilbert-Schmidt kernel operator $H$ with $\|H\|_2\leq 1$ satisfies $\|H-p\|_\square\leq 1/\sqrt{d}$ where $p$ is the constant function on $V\times V$ with value $\alpha=\int_{x,y}H(x,y)~dx~dy$.
\end{corollary}

\begin{proof} Let $\mathcal{H}^0$ denote the orthogonal space of the constant $1$ function on $V$. The smallest finite dimensional $G$ invariant subspace in $\mathcal{H}^0$ has dimension at least $d$. It is easy to see that $\ran(H-p)\subseteq\mathcal{H}_0$. Then corollary \ref{qra1} finishes the proof.
\end{proof}

We demonstrate the usefulness of this simple fact on the next example.
Let $S_n$ denote the $n$-dimensional sphere with the isometry invariant probability measure.
We call a graphon $W:S_n\times S_n\rightarrow [0,1]$ isometry invariant if $W$ is invariant under the induced action of the orthogonal group $O_{n+1}$ on $S_n$.
Note that a graphon $W$ is isometry invariant if and only if the value $W(x,y)$ depends only on the distance of $x$ and $y$.

The next proposition says that on a very high dimensional sphere every isometry invariant graphon is very close to being quasirandom.

\begin{proposition} If $W$ is an isometry invariant graphon with edge density $p$ on $S_n$ then $\|W-p\|_\square\leq 1/\sqrt{n+1}$.
\end{proposition}

\begin{proof} The smallest non trivial representation of the orthogonal group $O_{n+1}$ which appears on $L_2(S_n)$ has dimension $n+1$. Then corollary \ref{qra2} completes the proof.
\end{proof}

\subsection{Weakly random group actions}

\begin{definition} Let $G\subseteq U(L_2(V))$ be a unitary operator group. We denote by $\mathcal{I}(G,d)$ the set of $G$ invariant self adjoint kernel operators with $L_\infty$ norm at most $d$.
\end{definition}

\begin{lemma}\label{closedweak} The set $\mathcal{I}(G,d)$ is closed under weak convergence.
\end{lemma}

\begin{proof} Let $\{H_i\}_{i=1}^\infty$ be a weakly convergent sequence of kernel operators in $\mathcal{I}(G,d)$ and let $H$ be the weak limit. For every two functions $f,g\in L_2(V)$ we have that $g*H^\alpha f=\lim g^*H_i^\alpha f=\lim g^*H_if=g^*Hf$. This means that $H^\alpha=H$.
\end{proof}

We will need the following lemma about weak convergence.

\begin{lemma}\label{compconv} Let $C\subseteq L_2(V)$ be a compact set and $\{H_i\}_{i=1}^\infty$ be a sequence of kernel operators, with uniformly bounded $L_2$ norms, weakly converging to the $0$ function. Then
$$\lim_{i\to\infty}\max_{f,g\in C}\|g^* H_i f\|_2=0.$$
\end{lemma}

\begin{proof} Using compactness of $C$ we can choose sequences $\{g_i\}_{i=1}^\infty$ and $\{f_i\}_{i=1}^\infty$ in $C$ such that $\|g_i^*H_i f_i\|_2=\max_{f,g\in C}\|g^* H_i f\|_2=m_i.$ Assume that $m=\limsup_{i\to\infty} m_i>0$. Then by choosing a subsequence we can assume that $m=\lim_{i\to\infty} m_i$. Furthermore by compactness of $C$ we can assume by choosing a subsequence that $\lim_{i\to\infty} f_i=f$ and $\lim_{i\to\infty} g_i=g$ where the convergence is in the $L_2$ norm. Now using the fact that the $L_2$ norms of $H_i$ are bounded we obtain that $\lim_{i\to\infty} \|g^*H_i f\|_2=m>0$ which is a contradiction.
\end{proof}

\begin{definition} Let $G$ be a group of unitary operators on $L_2(V)$. We say that $G$ acts weakly random if for every natural number $n$ the space $U_n$ generated by all $G$ invariant subspaces in $L_2(V)$ of dimension at most $n$ is finite dimensional.
\end{definition}

Next theorem shows a surprising graph theoretic aspect of weakly random group actions.

\begin{theorem}\label{weakrand} Let $G$ be a weakly random operator group and let $\mathcal{I}(G,d)$ denote the set of self adjoint integral kernel operators $M:V\times V\rightarrow\mathbb{C}$ that are invariant under $G$ and have $L_\infty$ norm at most $d<\infty$. Then weak convergence on $\mathcal{I}(G,d)$ coincides with convergence in the cut norm.
In particular $\mathcal{I}(G,d)$ is cut norm compact.
\end{theorem}

\begin{proof} Let $\{H_i\}_{i=1}^\infty$ be a weakly convergent sequence in $\mathcal{I}(G,d)$. By subtracting (from every term) the weak limit $H$ (which is also in $\mathcal{I}(G,d)$ by lemma \ref{closedweak}) we get a sequence in $\mathcal{I}(G,2d)$ which converges to $0$ weakly. This means that without loss of generality we can assume that the weak limit of $\{H_i\}_{i=1}^\infty$ is the $0$ function. Let us choose an arbitrary real number $\epsilon>0$.
Let $U$ be the space generated by all $G$ invariant subspaces of dimension at most $d^2/\epsilon^2$ and let $U^0$ be the unit ball in $U$. Using lemma \ref{inveigen} we get that every normalized eigenvector of $H_i$ corresponding to an eigenvalue of absolute value bigger then $\epsilon$ is in $U^0$.
On the other hand, using the compactness of $U_0$ and lemma \ref{compconv} we get that there in an index $j$ such that if $i>j$ then $\max_{f\in U^0}\|f^*H_if\|_2<\epsilon$.
This means that if $i>j$ then the spectral radius of $H_i$ is at most $\epsilon$ and so $\|H_i\|_\square\leq\epsilon$.
\end{proof}

\subsection{Sphere vs. circle}

Let $f:[-1,1]\rightarrow\mathbb{R}$ be a bounded measurable function. We denote by $S(n,f)$ the graphon defined on the unit spehere $S_n=\{x|~x\in\mathbb{R}^{n+1},\|x\|_2=1\}$ with the uniform distribution such that $w(x,y)=f(xy)$ where $xy$ is the usual scalar product.

The underlying topological space (in the sense of \cite{LSz7}) of $S(n,f)$ is either the sphere $S_n$ or just one point if $f$ is constant.
In this part we point out that the case $n=1$ is very different for $n>1$.
Let $\mathcal{S}^0_n$ denote the subset in $S(n,f)$ where $0\leq f(x)\leq 1$ for every $-1\leq x\leq 1$.

\begin{proposition}\label{spherecom} Let $n\geq 2$. Then the set $\mathcal{S}^0_n$ is compact in the cut norm.
\end{proposition}

\begin{proof} For the first part let $O_{n+1}$ be the orthogonal group acting on $S_n$. The induced action of $O_n$ on $L_2(S_n)$ is defined by $f^\alpha(x)=f(x^\alpha)$ where $\alpha\in O_{n+1}$.
It is clear that spherical graphons of dimension $n$ are invariant under this action.
The representation theory of $O_{n+1}$ on $L_2(S_n)$ is a classical theory.
It acts weakly randomly which proves the first part.
\end{proof}

\begin{corollary} Let $n\geq 2$. Then the set $\mathcal{S}^0_n$ is closed in $\delta_\square$ and so it is characterizable by inequalities in subgraph densities .
\end{corollary}

\begin{proof} If $\{M_i\}_{i=1}^\infty$ in $\mathcal{S}^0_n$ is $\delta_\square$ convergent then we can choose a weakly convergent subsequence. By proposition \ref{spherecom} this subsequence is cot norm convergent. The limit $M$ is in $\mathcal{S}^0_n$ and it has to coincide with the $\delta_\square$ limit of $\{M_i\}_{i=1}^\infty$.
\end{proof}

\begin{proposition} If $n=1$ then the set $\mathcal{S}^0_n$ is not compact in the cut norm.
\end{proposition}

\begin{proof} Let us define the graphon $W$ on the circle $S_1$ by $W(x,y)=1$ if $xy>0$ and $W(x,y)=0$ if $xy<0$. Let us represent $S_1$ and the abelian group $A=\mathbb{R}/\mathbb{Z}$. It is easy to see an well known that for every fixed $k$ the map $\psi_k:a\mapsto ka$ is a measure preserving map on $A$. This means that $\delta_\square(W^{\psi_k},W)=0$ for every $k$.
We show that the sequence $\{W^{\psi_k}\}_{k=1}^\infty$ does not have a cut norm convergent sub sequence. Assume by contradiction the $\{W^{\psi_{k_i}}\}_{i=1}^\infty$ is convergent in the cut norm. Then the limit $L$ has $\delta_\square$ distance $0$ form $W$ and thus $\|L\|_2=\|W\|_2$.
By lemma \ref{corhv2} this means that $\{W^{\psi_{k_i}}\}_{i=1}^\infty$ is convergent in $L_2$.
It is easy to see that this is not the case.
\end{proof}

Without proof we mention that the set $\mathcal{S}^0_1$ is not even closed in $\delta_\square$. There are examples where a sequence of graphons in $\mathcal{S}^0_1$ converges to a graphon whose underlying topological space is the torus. Such a graphon can't be represented on the circle.

Motivated by the above results it is natural to introduce the following notion.

\begin{definition} Let $(V,\mathcal{A},\mu)$ be a probability space with $\sigma$-algebra $\mathcal{A}$ and measure $\mu$. Let $\mathcal{B}\subseteq \mathcal{A}\times\mathcal{A}$ be a sub $\sigma$-algebra on the product space $V\times V$. We call $\mathcal{B}$ weakly random if in the set of functions $\{M|M\in L_\infty(\mathcal{B}),\|M\|_\infty\leq 1\}$ weak convergence implies cut norm convergence.
\end{definition}

Proposition \ref{spherecom} says that if $V$ is the sphere $S_2$ and $\mathcal{B}$ consists of those Borel measurable sets that are invariant under the diagonal action of $O_3$ on $V\times V$ then $\mathcal{B}$ is weakly random.

\begin{question} Is there any characterization of weakly random $\sigma$-algebras?
\end{question}

Another interesting topic is to understand when $\{M|M\in L_\infty(\mathcal{B}),\|M\|_\infty\leq 1\}$ is closed in the $\delta_\square$ metric. Weakly randomness implies this but the other direction is not true.

\vskip 0.2in

\noindent
Bal\'azs Szegedy
\noindent
University of Toronto, Department of Mathematics,
\noindent
St George St. 40, Toronto, ON, M5R 2E4, Canada

\end{document}